\documentclass[letterpaper,12pt]{amsart}
\usepackage{hyperref}
\usepackage{latexsym}
\usepackage{amssymb}
\usepackage[cp850]{inputenc}
\usepackage{epsfig}
\usepackage{amsthm}
\usepackage{amscd}
\usepackage{amsmath}
\usepackage{amsfonts}
\usepackage{graphics}
\usepackage[all]{xy}

\newtheorem{theorem}{Theorem}
\newtheorem{lemma}[theorem]{Lemma}
\newtheorem{corollary}[theorem]{Corollary}
\newtheorem{prop}[theorem]{Proposition}

\newtheorem*{theorem*}{Theorem}
\newtheorem*{corollary*}{Corollary}
\theoremstyle{definition}

\newtheorem{definition}[theorem]{Definition}
\newtheorem*{remark*}{Remark}
\newtheorem{question}[theorem]{Question}
\newtheorem{observation}[theorem]{Observation}
\newtheorem*{definition*}{Definition}
\newtheorem*{example*}{Example}

\numberwithin{theorem}{section}

\newcommand{\BC}{\mathbb C} 
\newcommand{\BR}{\mathbb R} 
\newcommand{\BN}{\mathbb N} \newcommand{\BQ}{\mathbb Q}
\newcommand{\BS}{\mathbb S} \newcommand{\BZ}{\mathbb Z}
 \newcommand{\BT}{\mathbb T}

\newcommand{\BX}{\mathbb X}

 \newcommand{\CB}{\mathcal B}
\newcommand{\CC}{\mathcal C} \newcommand{\calD}{\mathcal D}
 
\newcommand{\CG}{\mathcal G}

 \newcommand{\CR}{\mathcal R}
\newcommand{\CS}{\mathcal S} \newcommand{\CT}{\mathcal T}
\newcommand{\CU}{\mathcal U}

\newcommand{\aut}{\textup{Aut}(F_n)}

\newcommand{\wt}{\widetilde}

\newcommand{\nid}{\noindent}

\newcommand{\tup}{\textup}

\DeclareMathOperator{\trace}{Tr}

\DeclareMathOperator{\rank}{rank}

\newcommand{\comment}[1]{}

\usepackage{etoolbox}

\makeatletter
\patchcmd{\epigraph}{\@epitext{#1}}{\itshape\@epitext{#1}}{}{}
\makeatother

\begin{document}
\title  [Non--solvable representations]  {Non virtually solvable subgroups of mapping class groups have non virtually solvable representations}
\author   {Asaf Hadari}
\date{\today}
\begin{abstract} Let $\Sigma$ be a compact orientable surface of finite type with at least one boundary component. Let $\Gamma \leq \tup{Mod}(\Sigma)$ be a non virtually solvable subgroup. We answer a question of Lubotzky by showing that there exists a finite dimensional homological representation $\rho$ of $\tup{Mod}(\Sigma)$ such that $\rho(\Gamma)$ is not virtually solvable. We then apply results of Lubotzky and Meiri to show that for any random walk on such a group the probability of landing on a power, or on an element with topological entropy $0$ both decrease exponentially in the length of the walk.\end{abstract}

\maketitle

\section{introduction}

Let $\Sigma$ be an orientable compact surface of finite type, possibly with boundary. The mapping class group of $\Sigma$, or $\tup{Mod}(\Sigma)$, is the group of orientation preserving diffeomorphisms of $\Sigma$ that fix the boundary point wise, up to isotopies that fix the boundary point wise. A number of years ago, Alexander Lubotzky asked us the following question. 

\begin{question} \label{lq} \tup{(Lubotzky)} Given a non virtually solvable subgroup $\Gamma < \tup{Mod}(\Sigma)$, is there a finite dimensional representation $\rho: \tup{Mod}(\Sigma) \to \tup{GL}(V)$ such that $\rho(\Gamma)$ is not virtually solvable?  
\end{question}

As we  discuss below, Lubotzky's question was asked in the context of trying to describe probabilistic properties of elements in subgroups of $\tup{Mod}(\Sigma)$. However, the question fits into a larger meta-question.

\begin{question} \label{urq} Which properties of the mapping class group can we discern in its representations? 
\end{question}

The answer to this larger question is important whenever one wants to apply representation theory to study mapping class groups. This larger question is quite difficult to answer. For instance it is not known in general 
whether or not mapping class groups are linear, that is - whether or not they have any faithful finite dimensional representations. 

In this paper, we provide a positive answer to Lubotzky's question in the case where $\Sigma$ has non-empty boundary. The class of representations we consider are called \emph{homological representations}, which we now describe. Pick a point $* \in \partial \Sigma$. Since elements of $\tup{Mod}(\Sigma)$ are diffeomorphisms that fix $*$ up to isotopies that fix $*$, we get an injection $\tup{Mod}(\Sigma) \hookrightarrow \tup{Aut}(\pi_1(\Sigma, *))$. Let $p: \Sigma' \to \Sigma$ be a finite cover corresponding to the subgroup $K < \pi_1(\Sigma, *)$. The subgroup $\tup{Mod}_K(\Sigma) = \{f \in \tup{Mod}(\Sigma): f(K)= K \}$ is a finite index subgroup of $\tup{Mod}(\Sigma)$. This gives a map: $\tup{Mod}_K(\Sigma) \hookrightarrow \tup{Aut}(K)$. There is a natural map $\tup{Aut}(K) \to \tup{Aut}(H_1(K, \BC)) \cong \tup{GL}(H_1(K, \BC))$ given by the taking the induced action of  a homomorphism of $K$ on $\frac{K}{[K, K]} \otimes \BC$. The composition of the above maps gives a homomorphism $\rho_p: \tup{Mod}_K(\Sigma) \to \tup{GL}(H_1(K, \BC))$ called a \emph{homological representation}. Inducing this representation to $\tup{Mod}(\Sigma)$ gives a homological representation of $\tup{Mod}(\Sigma)$. We prove the following result.

\begin{theorem} \label{theorem1}
Let  $\Sigma$ be a compact orientable surface of finite type with at least one boundary component. Let $\Gamma \leq \tup{Mod}(\Sigma)$ be a non virtually solvable subgroup. Then there exists a finite cover $p: \Sigma' \to \Sigma$  such that the image $\rho_p(\Gamma)$ of $\Gamma$ under the homological representation $\rho_p$ is not virtually solvable. Furthermore, if $\Gamma$ contains a pseudo-Anosov element then this cover can be taken to be a regular cover with a solvable deck group.

\end{theorem}

\nid One key element in our proof is our proof in \cite{eigoffuc} of the following result, which answers a question of McMullen (\cite{Mcm}) and also addresses Question \ref{urq}. 

\begin{theorem} \label{specrad} \tup{(Hadari)} Let  $\Sigma$ be a compact orientable surface of finite type with at least one boundary component. Let $f\in \tup{Mod}(\Sigma)$ be a mapping class with positive topological entropy. Then there exists a finite cover $p: \Sigma' \to \Sigma$ such that $\rho_p(f)$ has eigenvalues off the unit circle. Furthermore, if $f$ is pseudo-Anosov then this cover can be taken to be a regular cover with a solvable deck group. 
\end{theorem}

\nid Nearly concurrently with \cite{eigoffuc},  Yi Liu provided a different proof of a similar result to Theorem \ref{specrad} (see \cite{YiLiu}). Liu's proof covers the case where $\Sigma$ is a closed surface as well. However, his proof doesn't provide a cover with a solvable deck group. The solvability of the deck group is an important requirement for the proof appearing in this paper. 

\subsection{Theorem \ref{theorem1} and random elements in subgroups of $\tup{Mod}(\Sigma)$} Let $G$ be a finitely generated group, and let $S$ be a finite symmetric generating set. Let $\mu$ be the uniform measure on $S$. For any $X \subset G$ and any $k > 0$ let $\tup{prob}_k(X) = \mu^{\star k}(X)$, that is - the probability that the random walk defined by $S$ lands in $X$ after k steps. We say that the set $X$ is \emph{exponentially small in G} if there exists numbers $C > 0, b>1 $ such that $\tup{prob}_k(X) < Cb^{-k}$ for every $k$. 

In \cite{LMeiri1}, Lubotzky and Meiri use sieve methods on groups to study exponentially small sets in non-solvable groups. They then apply these results to mapping class groups in \cite{LMeiri2}, and to $\tup{Aut}(F_n)$ in \cite{LMeiri3}.  The motivation for asking question \ref{lq} was to extend this study to arbitrary subgroups of the mapping class group. One of the main results in their paper \cite{LMeiri1} is the following.  

\begin{theorem} \label{pes} \tup{(Lubotzky, Meiri)} If $G < \tup{GL}_n(\BC)$ is finitely generated and not virtually solvable then the set $\bigcup_{m=2}^\infty G^m$ is exponentially small in G. 
\end{theorem}

\nid Since homomorphisms of groups send powers to powers, an immediate corollary of Theorems \ref{theorem1} and \ref{pes} is the following. 

\begin{corollary} \label{pesmcg} Let $\Sigma$ be as in Theorem \ref{theorem1} and let $\Gamma < \tup{Mod}(\Sigma)$ be a non virtually solvable finitely generated group. Then $\bigcup_{m=2}^\infty \Gamma^m$ is exponentially small in $\Gamma$. 

\end{corollary}

As part of their proof of Theorem \ref{pes}, Lubotzky and Meiri use a result of Bourgain, Gumburd, and Sarnak to prove that if $G < \tup{GL}_n(\BQ)$ is a non virtually solvable finitely generated with semi-simple Zariski closure then the set of virtually unipotent elements is exponentially small in $G$. 

 Given a finite index subgroup $K < \pi_1(\Sigma, *)$, the group $\tup{Mod}_K(\Sigma)$ acts on $H_1(K, \BZ)$. Thus, the image $\rho_K(\Gamma)$ is contained in $\tup{GL}(H_1(K, \BQ))$. Let $\CG$ be the identity component of the Zariski closure of this group. The group $\CG$ may not be semisimple. Dividing by its radical, we get a map $\alpha: \CG \to \tup{GL}_m$ for some $m$. The composition $\beta: \Gamma \to \tup{GL}(H_1(K, \BQ)) \to \tup{GL}_m$ has image in $\tup{GL}_m(\BQ)$ because the radical of $\CG$ is defined over $\BQ$. The identity component of the Zariski closure of $\beta(\Gamma)$ is semisimple. Furthermore, since the map $\beta$ is given by taking a quotient of the adjoint representation, it sends virtually unipotent elements to virtually unipotent elements. 
 
 Every element of the mapping class group is either finite order, has a power that is a multi twist, or has positive topological entropy. The image of a multi twist under any homological representation is easily seen to be virtually unipotent. Thus, Lubotzky and Meiri's proof, together with our Theorem \ref{theorem1} give the following. 
 
 \begin{corollary}
 Let $\Sigma$ be as in Theorem \ref{theorem1} and let $\Gamma < \tup{Mod}(\Sigma)$ be a non virtually solvable finitely generated group. Let $X \subset \Gamma$ be the set of all elements with topological entropy $0$. Then $X$ is exponentially small in $\Gamma$. \end{corollary}

\subsection{Sketch of the proof of Theorem \ref{theorem1}} Throughout this paper, we use the letter $\CB$ to denote the group of invertible upper triangular matrices, and $\CU$ to be the group of invertible upper triangular matrices with $1$'s along the diagonal (where the size of the matrices is taken to be understood by the context). The main philosophy of the argument is the following. A virtually solvable linear group has a finite index subgroup that can be conjugated into $\CB$. The commutator group $[\CB,\CB]$ is $\CU$, and hence every eigenvalue of every element of $[\CB,\CB]$ is $1$. Theorem \ref{specrad} is a tool for finding representations where given elements have eigenvalues off the unit circle, and hence it should be useful in proving theorem \ref{theorem1}. 

Let $\rho$ be the homological representation corresponding to the trivial cover $\Sigma \to \Sigma$. If $\rho(\Gamma)$ is not virtually solvable, then we are done. Otherwise, $\Gamma$ has a finite index subgroup $\Gamma'$  such that $\rho(\Gamma')$ can be conjugated into $\CB$. Let $\Gamma_0 = [\Gamma', \Gamma']$. We call a cover $p: \Sigma' \to \Sigma$ corresponding to the subgroup $K < \pi_1(\Sigma, *)$ \emph{$\Gamma_0$-admissible} if $K$ is the pullback of a finite index subgroup under the map $\pi_1(\Sigma, *) \to H_1(\Gamma_0 \rtimes \pi_1(\Sigma, *), \BZ) / \tup{(torsion)}$. The proof involves two major steps. 

\begin{enumerate}
\item In Section \ref{hriac}, we show that if Theorem \ref{theorem1} is false, then $\rho_p(\Gamma_0)$ can be conjugated into upper triangular matrices, for every $\Gamma_0$-admissible cover $p$. This is the content of Proposition \ref{prop1}. The important part here is that $\rho_p(\Gamma_0)$ itself, and not just a finite index subgroup, can be conjugated into $B$. This involves a careful study of homological representations corresponding to abelian covers, together with an argument we call the rigidity argument that shows that if $\rho_p(\Gamma_0)$ can be conjugated into $\CB$ for some admissible $p$ then it can be conjugated into $\CB$ for all admissible $p$.   

\item The group $[\Gamma_0,\Gamma_0]$ contains an element $h$ of positive topological entropy.We show that if this element if pseudo-Anosov, and Theorem \ref{theorem1} is false, then argument from our proof of Theorem \ref{specrad} can be modified to produce a tower of $\Gamma_0$-admissible covers $\Sigma_k \to \ldots \to \Sigma_0 = \Sigma$ such that $\rho_{\Sigma_k \to \Sigma}(h)$ has eigenvalues off the unit circle. This is the content of Proposition \ref{eigoffuc}. The result now follows after some reduction steps, which are covered in the final section.

\end{enumerate}

\subsection{Acknowledgements} We would like to thank Alex Lubotzky for bringing this question to our attention, and Thomas Koberda for interesting discussions about this result. 

\section{Homological representations in admissible covers} \label{hriac}

\begin{definition} Let $F_n \cong \pi_1(\Sigma, *)$, and let $\Gamma_0 < \tup{Aut}(F_n)$. A  subgroup $K \leq F_n$ is called \emph{$\Gamma_0$-admissible} if it is the pullback of a finite index subgroup by the composition $F_n \to \Gamma_0 \rtimes F_n \to H_1(\Gamma_0 \rtimes F_n, \BZ) /\tup{(torsion)}$. 
\end{definition}

\nid Our goal in this section is to prove the following. 

\begin{prop} \label{prop1} Suppose $(\Gamma_0)_* \leq \tup{GL} (H_1(F_n, \BC))$ is conjugate to a subgroup of $\CB$ and that $\rho_p(\Gamma_0)$ is virtually solvable for every $\Gamma_0$-admissible $p$.  Then $\rho_p(\Gamma_0)$ can be conjugated into $\CB$ for every $\Gamma_0$-admissible $p$. 

\end{prop}

\subsection{The chain space of an abelian covers}
Let $R = \bigvee_1^n S^1$ be the graph given by a wedge of $n$ circles, and let $v \in R$ be its vertex.  We have that $\pi_1(R, v) \cong F_n$. Pick a preferred orientation on each edge of $R$. Given a regular cover $p: \wt{R} \to R$, denote by $\CC = \CC_p = C_1(\wt{R}, \BC)$ the space of simplicial $1$-chains in $\wt{R}$ with coefficients in $\BC$. 

Let $G = G_p$ be the deck group of the covering space $p$. The action of $G$ on $\wt{R}$ by deck transformations gives $\CC$ the structure of a $G$-module. This structure has a simple description. 

Pick a lift $\wt{v}$ of $v$ to $\wt{R}$. Every other lift of $v$ is an image of $\wt{v}$ under a unique deck transformation. This identifies the vertex set $V(\wt{R})$ with the deck group $G$. Given a positively oriented edge $e$ of $R$, every lift $\wt{e}$ of $e$ is the image under a deck transformation of an edge originating at $\wt{v}$. This identifies the lifts of $e$ with $G$. Write $E(R) = \{e_1, \ldots, e_n\}$. Elements of $\CC$ are formal combinations of oriented edges in $\wt{R}$. We can write such an element as the sum of a formal combination of lifts of $e_1$, a formal combination of lifts of $e_2$, etc. In this way, we have that as $G$-modules: $\CC \cong (\BC[G])^n  $.

Suppose that the group $G$ is a finite abelian group. Let $\xi: G \to \BC^*$ be a character of $G$. Let $$\CC[\xi] = \{x \in \CC: gx = \xi(g) x, \forall g \in G \}$$
The space $\CC[\xi]$ is an $n$-dimensional subspace of $\CC$. Indeed, let $\wt{e}_1, \ldots, \wt{e}_n$ be the lifts of $e_1, \ldots, e_n$ originating at $\wt{v}$. Since every edge in $\wt{R}$ is a $G$-image of one of these edges, any $x \in \CC[\xi]$ is completely determined by its coefficients for the edges $\wt{e}_1, \ldots, \wt{e}_n$. 

Since $G$ is a finite abelian group, its group ring decomposes as a direct sum of all of its characters. Thus, we get: 

$$\CC = \bigoplus_{\xi} \CC[\xi] $$

\subsection{The Magnus representation and specializations of the Magnus matrix} Let $p: (\wt{R}, \wt{v}) \to (R,v)$ be an abelian cover as above corresponding to the subgroup $K < F_n$. Suppose that $\Gamma_0(K) = K$. Then we have a well defined map $\Gamma_0 \to \tup{GL}(\CC_p) \cong \tup{GL}(\BC[G]^n)$. Let $\wt{e}_1, \ldots, \wt{e_n}$ be the preferred lifts of $e_1, \ldots, e_n$ as above. Given $\gamma \in \Gamma_0$, let $M_\gamma \in GL(\BC[G]^n)$ be the matrix whose $i^{th}$ column is the element of $\BC[G]^n$ corresponding to $\gamma \cdot (\wt{e}_i)$. 

Since $K$ is $\Gamma_0$-invariant, the group $\Gamma_0$ acts on the deck group $G \cong F_n/K$. This induces a map $\Gamma_0 \to \tup{End}(\BC[G]^n)$. Denote the image of $\gamma \in \Gamma_0$ in $\tup{End}(\BC[G]^n)$ by $\sigma(\gamma)$. 

Let $\wt{e}_{i,g}$ be the lift of $e_i$ corresponding to the element $g \in G$. We identity $\wt{e}_{i,g}$ with the $1$-chain $1 \cdot \wt{e}_{i,g}$. Then $\gamma(\wt{e}_{i,g}) = \sigma(\gamma) M_\gamma \wt{e}_{i,g} $.  The assignment $\gamma \to M_\gamma$ is not a homomorphism, but it is a crossed homomorphism. When $\Gamma_0$ induces the trivial map on $G$, then the assignment becomes a homomorphism. 

If $p: \wt{R} \to R$ is the universal abelian cover (that is - the cover corresponding to the commutator subgroup $[F_n, F_n] <F_n$ then the matrix $M_\gamma$ is called the \emph{Magnus matrix associated to $\gamma$}. If we consider $H_1(F_n, \BZ) \cong \BZ^n$ as a multiplicative group generated by $X_1, \ldots, X_n$ then we have $M_\gamma \in M_n(\BC[X_1^{\pm 1}, \ldots, X_m^{\pm 1}])$. The entries of this matrix are rational functions in the variables $X_1, \ldots, X_n$. 

Given an $n$-tuple $\overline{\xi} = (\xi_1, \ldots, \xi_n) \in  \BC^n$, we can define the \emph{specialization of $M_\gamma$ at $\overline{\xi}$ or $M_\gamma(\overline{\xi})$} to be the element of $M_n(\BC)$ obtained by plugging in $\xi_i$ for $X_i$ in every entry of $M_\gamma$. The tuple $\overline{\xi}$ also gives a character $ \cong H_1(F_n, \BZ) \to \BC^*$ obtained by sending $X_i$ to $\xi_i$. 

Given a finite abelian cover $q: R' \to R$, with deck group $G$, any character $\xi$ of $G$ determines a character of $H_1(F_n, \BZ)$. Suppose that the map $\gamma$ lifts to the cover $R'$. Let $V = H_1(R', \BZ)^*$ be the unitary dual of $H_1(R', \BZ)$. The automorphism $\gamma$ induces a map $\gamma^*: V \to V$. For any  character $\xi \in V$, $\gamma$ sends the space $\CC[\gamma^*\xi]$ to $\CC[\xi]$ via the map $M_\gamma(\gamma^* \xi)$. This gives us the following description of the map $\gamma_*: \CC_q \to \CC_q$. 

\begin{observation} \label{blockmatrix}

Using the decomposition $\CC_q = \bigoplus_{\xi} \CC_q[\xi]$, we can pick a basis in which  the map $\gamma_*: \CC_q \to \CC_q$ can be written block matrix form, where the $n \times n $ block in the $(\gamma^* \xi, \xi)$ spot is $M_\gamma(\gamma^* \xi)$ and all other blocks are zero blocks. 
 
 \end{observation}

\nid For further information about the Magnus representations and its specializations, see \cite{eigoffuc}, \cite{mag}, \cite{mag2}. 

\subsection{The rigidity argument} Let $n$ be an integer and let \\ $R = \BC[X_1^{\pm1}, \ldots, X_n^{\pm 1}]$. Let $\CB < \tup{GL}_n(\BC)$ be the group of invertible upper triangular matrices and $\CU < \tup{GL}_n(\BC)$ be the group of all upper triangular matrices with $1$'s along the diagonal. Let $\BT = \{(\xi_1, \ldots, \xi_n) \in \BC^n : |\xi_i| = 1, \forall i  \}$

\begin{lemma} \label{boundpowers}
There exists a function $f: \BN \to \BN$ such that every finitely generated $G < \tup{GL}_n(\BC)$ which contains a normal abelian subgroup of finite index consisting of diagonalizable matrices contains such a subgroup of index at most $f(n)$.
\end{lemma}

\begin{proof}
	Let $A \lhd G$ be a normal diagonalizable abelian subgroup of finite index that is maximal under inclusion. Suppose first that all the elements of $A$ are multiples of the identity. 
	
	Pick a generating set $G = \langle X_1, \ldots, X_m \rangle$. Pick $\delta_i$ such that $X_i' = \delta_i X_i$ satisfies $\tup{det}(X_i') =1$. Let $G' = \langle X_1', \ldots, X_m' \rangle$. If we find a number $M$ such that $[P^M, Q^M] = I_n$ for any $P,Q \in G'$ and that $P^M$ is diagonalizable for all $P \in G'$ Then the same will hold for $G$. Thus, it is enough to assume that $G < \tup{SL}_n(\BC)$. In this case, every element of $A$ must have finite order (since they are multiples of the identity matrix), and thus $G$ is finite. The result now follows from the Jordan-Schur theorem, which states that for every $n$ there is a number $g(n)$ such that any finite subgroup of $\tup{GL}_n(\BC)$ has a normal abelian subgroup of index at most $g(n)$.
	
	We now turn to the case where $A$ contains elements that are not multiples of the identity. Decompose $\BC^n$ as a direct sum of characters of $A$: $\BC^n \cong \bigoplus_i m_i \chi_i$, where $m_i$ is the multiplicity of the character $\chi_i$. 
	
	As $A \lhd G$, any element $X$ of $G$ acts by permutations on the set $\{m_i \chi_i\}_i$. Thus, there is a number $M$, dependent only on $n$ such that $X^M$ fixes each such representation. For each $i$, the restriction $A|_{m_i \chi_i}$ gives a group of matrices that are multiples of the identity map on the space $m_i \chi_i$. The first part of the proof now gives the result.

\end{proof}

\begin{lemma} \label{deformation} Let $\BS \subset \BT$ be an algebraic sub-torus. Let $a_1, \ldots a_m \in \tup{GL}_n(R)$. Suppose that $\langle a_1(\xi), \ldots,  a_m(\xi) \rangle$ is virtually solvable for every $\xi \in \BS$ and that there exists $\zeta \in \BS$ such that $\langle a_1(\zeta), \ldots, a_m(\zeta) \rangle$ is conjugate to a subgroup of $\CU$. Then $\langle a_1(\xi), \ldots,  a_m(\xi) \rangle$ is conjugate to a subgroup of $\CB$ for every $\xi \in \BS$. 

\end{lemma}

\begin{proof}
Let $\xi \in \BS$, and let $G = G(\xi) = \langle a_1(\xi), \ldots,  a_m(\xi) \rangle$. The group $G$ has a finite index subgroup $G_0 \lhd G$ that is conjugate to a subgroup of $\CB$. Indeed, let $\CG$ be the Zariski closure of $G$. The group $\CG$ is virtually solvable. Let $\CG_0$ be the connected component of the identity. As a connected solvable algebraic group, $\CG_0$ is conjugate to a subgroup of upper triangular matrices. Take $G_0 = \CG_0 \cap G$. We wish to prove that we can take $[G: G_0] = 1$.

Let $V$ be the subspace of all vectors $v \in \BC^n$ such that $v$ is an eigenvector for every element of $G_0$. The space $V$ is $G$-invariant, and the restriction map gives a homomorphism $r: G \to \tup{GL}(V)$ such that the image of $r(G_0)$ is a finite index, normal abelian subgroup of $r(G)$ consisting of diagonal matrices. By Lemma \ref{boundpowers}, up to a change of basis for $V$, there is a number $M$ which depends only on $n$ such that $r(g^M)$ is diagonal for every $g \in \Gamma$. 

Given a diagonalizable matrix $X$ and an integer $k$, a diagonalizing basis for $X^k$ can fail to be a diagonalizing basis for $X$ only if $X$ has two different eigenvalues $\lambda \neq \mu$ such that $\lambda^k = \mu^k$. Thus, if $r(G)$ is not  diagonal then at least one of the $a_i(\xi)$ has the property that it has eigenvalues $\lambda \neq \mu$ such that $\lambda^k = \mu^k$ for some $1 < k \leq M$. Now consider the space $\BC^n / V$, and proceed inductively as before. We deduce that if $[G:G_0] \neq 1$, then at least one of the $a_i(\xi)$ has the property that it has eigenvalues $\lambda \neq \mu$ such that $\lambda^k = \mu^k$ for some $1 < k \leq M$. 

Let $\BX \subset \BS$ be the set of all $\xi$ such that $G(\xi)$ can be conjugated into $\CB$, and let $\BX^c = \BS \setminus \BX$. Suppose there exists a matrix $X$ such that $Xa(\xi) X^{-1}$ and $X b(\xi) X^{-1}$ are both upper triangular. Using the polar decomposition on $\tup{GL}_n(\BC)$, we can find a matrix $Y \in U_n(\BC)$ such that $Ya(\xi) Y^{-1}$ and $Y b(\xi) Y^{-1}$ are both upper triangular. Given a sequence $\{\xi_i\}_i \subset \BX$ such that $\xi_i \to \xi_\infty$, compactness of $U_n(\BC)$ and passing to a subsequence gives that $\xi_\infty \in \BX$. Thus, $\BX$ is closed and $\BX^c$ is open in the Hausdorff topology on $\BS$. 

If $\BX = \BS$ we are done. Suppose that $\BX \neq \BS$. For every $k > 1$ and every $1 \leq j \leq m$ let $N_j(k)$ be the set of all $\xi$ such that $a_j(\xi)$ has eigenvalues $\lambda \neq \mu$ such that $\lambda^k = \mu^k$. For any $\xi$, let $p_k(\xi)$ be the characteristic polynomial of $a_j^k(\xi)$. The set $N_j$ is then the set of all $\xi$ such that:
  
$$\deg \tup{gcd}(p_1(\xi), p_1'(\xi)) <  \deg \tup{gcd}(p_k(\xi), p_k'(\xi))$$

\nid This can be restated as: 

$$\dim \ker p_1'(\xi) [a_j(\xi)] < \dim \ker p_k'(\xi)[a_j^k(\xi)] $$

\nid For any $r$, the set of all $\xi$ such that $$\dim \ker p_1'(\xi) [a(\xi)] = n - \rank \ker p_1'(\xi) [a(\xi)] \leq r$$ is Zariski open in the real Zariski topology on $\BS$, since it is given by a union of non-vanishing sets of determinants of $(n-j) \times (n-j)$ minors of $a_j(\xi)$. The same goes for the set of all $\xi$ such that $\dim \ker p_k'(\xi) [a_j^k(\xi)]  \leq r$. This implies that $N_a$ is a finite union of intersections of a Zariski open set with a Zariski closed set. Thus we have that $N_j \subset F_j$ for some Zariski closed $F_j$, such that $N_j$ is open in the induced Zariski topology on $F_a$. 

Since $\BX^c$ is open and contained in $N_1 \cup \ldots \cup N_m$, we have that one of the $F_j$'s is all of $\BS$ (since $\BS$ is irreducible). Thus one of $N_j$'s is dense (in the Hausdorff topology) in $\BS$. Assume it's $N_1$. 

Let $\CR$ be the space of unordered $\left( \begin{array}{c} n \\ 2 \end{array} \right)$-tuples of complex numbers. Let $\rho: \BS \to \CR$ be given by sending $\xi$ to the collection of all $\frac{\lambda}{\mu}$ where $\lambda, \mu$ are (not necessarily different) eigenvalues of $a_1(\xi)$. By the above, for any $\xi \in N_1$, the collection $\rho(\alpha)$ contains a root of unity $1 \neq \omega$ of degree at most $M$. The tuple $\rho(\zeta)$ contains only the number $1$. This is impossible, since $\rho$ is continuous and $1$ is not in the closure of the set of all non-1 roots of unity of degree at most $M$. Thus, $\BX = \BS$, as required.

\end{proof}

\nid We can now prove a stronger version of Lemma \ref{deformation}, that allows us to relax one of the conditions while still getting the same conclusion.

\begin{lemma} \label{deformation'} Let $\BS \subset \BT$ be an algebraic sub-torus. Let $a_1, \ldots a_m \in \tup{GL}_n(R)$. Suppose that $\langle a_1(\xi), \ldots,  a_m(\xi) \rangle$ is virtually solvable for every $\xi \in \BS$ and that there exists $\zeta \in \BS$ such that $\langle a_1(\zeta), \ldots, a_m(\zeta) \rangle$ is conjugate to a subgroup of $\CB$. Then $\langle a_1(\xi), \ldots,  a_m(\xi) \rangle$ is conjugate to a subgroup of $\CB$ for every $\xi \in \BS$. 

\end{lemma}

\begin{proof}

Let $R_0 = \BC[\BS]$ be the ring of functions on $\BS$. Let $G = \langle a_1, \ldots, a_m \rangle$, where each $a_i$ is thought of as an element of $\tup{GL}_n(R_0)$. By our assumption and by Lemma \ref{boundpowers}, this group is virtually nilpotent, and thus its commutator subgroup is finitely generated. Let $G'$ be the commutator subgroup of $G$. By our assumption on $\zeta$,  and the fact that $[\CB, \CB] = \CU$, we have that $G'(\zeta)$ can be conjugated into $\CU$. Apply Lemma \ref{deformation} to $G'$ to get that $G'[\xi]$ can be conjugated into $\CB$ for every $\xi \in \BS$. 

Using Lemma \ref{boundpowers}, we have that $G$ contains a finite index subgroup $H$ such that $H(\xi)$ can be conjugated into $\CB$ for every $\xi \in \CS$. Since $G'(\xi)$ can be conjugated into $\CB$, we can pick $H$ such that $A = G/H$ is finite and abelian. 

Given any $\xi \in \BS$ let $V_1(\xi)$ be the space of common eigenvectors of all the elements of $H(\xi)$. The space $V_1(\xi)$ is $G(\xi)$-invariant. Restriction to $V_1(\xi)$ gives a homomorphism $r_1: G(\xi) \to \tup{GL}(V_1(\xi))$. The group $G(\xi)$ acts on the space $\BC^n/V_1$. Let $V_2(\xi)$ be the subspace of all the common eigenvectors of the elements of $H(\xi)$ in $\BC^n/V_1$. Define the restriction map $r_2: G \to \tup{GL}(V_2)$ as above. Proceed inductively in this manner to define $V_i, r_i$ for every $i$.

The action of $G(\xi)$ on itself by conjugation gives a homomorphism $\gamma_i: G(\xi) \to \tup{Aut}(r_i(H(\xi)))$. By definition, $H(\xi) \leq \ker \gamma_i$ for every $i$. Suppose that the image of $\gamma_i$ is the trivial automorphism for every $i$. Since $A$ is abelian, and since abelian groups of matrices can be conjugated into $\CB$, we get that $G(\xi)$ can be conjugated into $\CB$. Thus, it is enough to prove that the image of each $\gamma_i$ is trivial, for every $\xi$. 

 Let $\calD \subset \tup{GL}_n(\BC)$ be the group of invertible diagonal matrices and let $D = \langle d_1, \ldots, d_k \rangle < \calD$ be a subgroup.  We can write $\BC^n = \bigoplus_{i=1}^s W_i$, where $W_i$ is a maximal subspace such that $d_j|_{W_i}$ is a multiple of $I_{W_i}$ for every $j$. The group $N(D)$, the normalizer of $D$, permutes the subspaces $W_i$. If an element $g \in N(D)$ fixes the subspace $W_i$, then the restriction of $g$ to $W_i$ commutes with every $d_j|_{W_i}$.

As in the proof of Lemma \ref{deformation}, let $\BX \subset \BS$ be the set of all $\xi$ such that $G(\xi)$ can be conjugated into $\CB$. The same proof shows that $\BX$ is closed in the Hausdorff topology on $\BS$. We now wish to show that it is also open. Since $\CS$ is connected, and since we assumed that $\BX \neq \emptyset$, this will conclude the proof. 

Let $\zeta_0 \in \BX$, and let $\xi \in \BS$. Without loss of generality, we can assume that $G(\zeta_0) \subset \CB$. Let $V_i, r_i$ be as above. Let $D$ be the restriction of $H(\xi)$ to $V$. Write $V_1= \bigoplus_{i=1}^s W_i$ as above. The group $r_1(G(\xi))$ normalizes $D$. 

For any norm on matrices, and any  $\epsilon > 0$, if $\xi$ is sufficiently close to to $\zeta_0$ then $r_1(a_j(\xi))$ has distance $ < \epsilon$ from an upper triangular matrix for every $j$. If $\epsilon$ is sufficiently small, such a matrix cannot permute the spaces $W_i$  non-trivially. Thus, $r_1(a_j(\xi))$ commutes with all elements of $D$, so $\gamma_1(G(\xi)) = 1$. Proceeding inductively in this manner, we get that $\gamma_i(G(\xi)) = 1$ for every $i$, and thus $\xi \in \BX$. This implies that $\BX$ is open, as required. 

\end{proof}

\nid We are now ready to prove the result we need. 

\begin{lemma} \label{deformation2}
 Let $\BS \subset \BT$ be an algebraic sub-torus. Let $a_1, \ldots a_m \in \tup{GL}_n(R)$. Let $Y \subset \BS$ be a dense set. Suppose that Suppose that $\langle a_1(\xi), \ldots,  a_m(\xi) \rangle$ is virtually solvable for every $\xi \in Y$ and that there exists $\zeta \in Y$ such that $\langle a_1(\zeta), \ldots, a_m(\zeta) \rangle$ is conjugate to a subgroup of $\CB$. Then $\langle a_1(\xi), \ldots,  a_m(\xi) \rangle$ is conjugate to a subgroup of $\CB$ for every $\xi \in \BS$. 
\end{lemma}

\begin{proof}
Let $G$ be as in the proof of Lemma \ref{deformation}. Suppose $G(\xi)$ is not virtually solvable. Let $M$ be the number given in the the argument of Lemma \ref{deformation}, such that for any virtual solvable $K < \tup{GL}_n(\BC)$, the group $K^M$ can be conjugated into $\CB$.  Since $G(\xi)$ is not solvable, there is some $I_n \neq \gamma$ in $(G^M)^{(n+1)}$, the $(n+1)^{th}$ term in the derived series of $G^M$. We can consider $\gamma$ as a word in the generators $a_1, \ldots, a_m$, and thus define $\gamma(\xi')$ for any $\xi'$. Since $\gamma(\xi) \neq I_n$, we have that $\gamma(\xi') \neq I_n$ in some open neighborhood of $\xi$. Thus, $G(\xi')$ cannot be virtually solvable for any such $\xi'$. This contradicts our assumption on the density of $Y$. Thus, $Y = \BS$ and we can apply Lemma \ref{deformation'} to prove the result. 

\end{proof}

\subsection{Proof of Proposition \ref{prop1}}

Given $G \leq \tup{Aut}(F_n)$, let $G_* \in \tup{Aut}(H_1(F_n, \BZ))$ be the induced homomorphism.

\begin{lemma} \label{uppertriangularmagnus} Suppose that  $G_*$ can be conjugated into $\CB$ and that $G$ has the property that the image  of $G$ in $\tup{GL}(\CC)$ is virtually solvable for abelian every cover to which $G$ lifts.  Let $\BT = H_1(F_n, \BZ)^*$ be the character torus of $F_n$, and let $\BS \subset \BT$ be the sub-torus of all $G$-invariant characters. For any $\xi \in \BS$, let $G(\xi)$ be the group generated by $M_\gamma(\xi)$, for all $\gamma \in G$. Then for any $\xi \in \BS$, the group $G(\xi)$ can be conjugated into the group of upper triangular $n \times n$ matrices. 

\end{lemma}

\begin{proof}
By Lemma \ref{deformation2}, it is enough to prove this for the set of all characters in $\BS$ with finite image (This set is dense in the torus $\BS$, which is defined over $\BQ$). The group $G(1)$ can be conjugated into upper triangular matrices by assumption. For any $\xi \in \CS$, the character $\xi$ is the character of the deck group of some admissible finite abelian cover $R' \to R$ to which $\Gamma$ can be lifted. By observation \ref{blockmatrix}, we can choose a basis such that every element of the image of $G$ in $\tup{GL}(H_1(R',\BC))$ can be written in block matrix form as a diagonal matrix with matrices of the form $G(\xi)$ along the diagonal. Thus, if the image of $G$ is virtually solvable, then so is each such $G(\xi)$. The result now follows from Lemma \ref{deformation2}.

\end{proof}

\begin{lemma} \label{nschains} If $p: R' \to R$ is an abelian cover to which $\Gamma$ lifts, and the image of $\Gamma$ in $\tup{GL}(H_1(\CC_p, \BC))$ is not solvable then neither is the image of $\Gamma$ in $\tup{GL}(H_1(R', \BC))$
\end{lemma}

\begin{proof}
Since every cycle in a graph is also a chain, we have a natural inclusion $H_1(R', \BC) \hookrightarrow C_1(R', \BC)$ as a $\Gamma_*$-invariant subspace. We begin by claiming that for every $h \in \Gamma$, $h_*: C_1(R', \BC) \to C_1(R', \BC)$ has eigenvalues off the unit circle if and only if $h_*: H_1(R', \BC) \to H_1(R', \BC)$ has eigenvalues off the unit circle. 

The if direction is obvious. We now show the only if direction. For brevity, denote $V = C_1(R', \BC)$ and $W = H_1(R', \BC)$.  The space $V$ is spanned by elements of the form $1 \cdot e$, where $e$ ranges over all edges of $R'$. Pick any norm $\|\cdot \|$ on $V$, and let $\lambda > 1$ be the spectral radius of the action of $h_*$ on $V$. Then there exists an edge $e$ such that $$\lim \sup_{j \to \infty} \frac{1}{j} \log \|h_* (e) \| = \log \lambda  $$
Denote $L_\lambda$ to be the direct sum of all generalized eigenspaces corresponding to eigenvalues with absolute value $\lambda$.  Setting $\nu_j = \frac{h_*^j (e)}{ \|h_*^j (e) \|}$, we have that the distance from $\nu_j$ to $L_\lambda$ goes to $0$ as $j \to \infty$.
Note that $\varphi^j(e)$ corresponds to a path in $R'$, and any such path can be closed to a loop using a bounded number of edges. Thus, the distance of $\nu_j$ from $W$ goes to $0$ as $j \to \infty$.  Therefore $L_\lambda \cap W \neq \{0 \}$. Since this space is $h_*$-invariant, it contains an eigenvector with eigenvalue of absolute value $\lambda$.

Now, assume that $\Gamma_*: W \to W$ is virtually solvable. Then $\Gamma$ contains a subgroup $K \lhd \Gamma$ such that $\Gamma / K$ is virtually solvable and $K_*: W \to W$ is trivial. Thus, every eigenvalue of every element of  $K_*: V \to V$ is on the unit circle. Since $\Gamma$ preserves the lattice of chains with coefficients in $\BZ$, the group $K_*$ can be conjugated into $\tup{GL}(C_1(R', \BZ))$. By a well known theorem of Kronecker, we have that all eigenvalues of all elements of $K_*$ are roots of unity. By passing to a finite index subgroup, we can assume that all eigenvalues of every element of $K$ is $1$.  By Tits' proof of the Tits alternative (\cite{TitsA}), a non-virtually solvable linear group contains a free group generated by two elements that have an eigenvalue of absolute value $>1$. Thus, $K_*$ must be virtually solvable. But this implies that $\Gamma_*: V \to V$ is also virtually solvable, which is a contradiction.

\end{proof}

\nid We are now able to prove Proposition \ref{prop1}.

\begin{prop} Suppose $(\Gamma_0)_*: H_1(F_n, \BC) \to H_1(F_n, \BC)$ is conjugate to a subgroup of upper triangular matrices and that $\rho_p(\Gamma_0)$ is virtually solvable for every $\Gamma_0$ admissible $p$.  Then $\rho_p(\Gamma_0)$ can be conjugated into the group of upper triangular matrices for every $\Gamma_0$-admissible $p$. 

\end{prop}

\begin{proof}
Since $p$ is an admissible cover, every character of the deck group of $p$ is $\Gamma_0$ invariant. Indeed, the deck group of $p$ is a subgroup of $H_1(\Gamma_0 \rtimes F_n)/ \tup{torsion}$, and the action of $\Gamma_0$ in this group $\Gamma_0 \rtimes F_n$ is given by conjugation and hence acts trivially on first homology. By Lemma \ref{nschains}, the image of $\Gamma_0$ in $\tup{GL}(H_1(\CC_p, \BC))$ must be virtually solvable. By Observation \ref{blockmatrix}  the image of $\Gamma_0$ in $\tup{GL}(H_1(\CC_p, \BC))$ can be written in block matrix notation as a diagonal matrix with blocks of the form $\Gamma_0(\xi)$ along the diagonal. By Lemma \ref{uppertriangularmagnus}, we can find a basis in which each such block is upper triangular. This proves the result. 
 
\end{proof}

\section{Finding eigenvalues off the unit circle using admissible covers}
\nid Our goal in this section is to prove the following proposition. 
\begin{prop} \label{eigoffuc}  Suppose $\Gamma_0 \leq \tup{Mod}(\Sigma)$ is finitely generated, not virtualy solvable and has the properties that $(\Gamma_0)_* \leq \tup{GL}(H_1(\Sigma, \BC))$ can be conjugated into $\CU$ and and that $\rho_p(\Gamma_0)$ is virtually solvable for every $\Gamma_0$ admissible $p$. Suppose $h \in [\Gamma_0, \Gamma_0]$ is a pseudo-Anosov mapping class. Then there exists a sequence of finite covers $\Sigma_k \to \ldots  \to \Sigma_0 = \Sigma$  such that the following hold: 
\begin{enumerate}
\item The group $\Gamma_0$ lifts to each cover. 
\item The cover $\Sigma_{i+1} \to \Sigma_i$ is $\Gamma_0$-admissible for every $i$. 
\item The endomorphism $h_*: H_1(\Sigma_k, \BC) \to H_1(\Sigma_k, \BC)$ has eigenvalues off the unit circle. 

\end{enumerate}
\end{prop}

\nid Our proof here is almost entirely identical to our proof of the following Theorem \ref{specrad}. One of the central notions in our proof of \ref{specrad} is the notion of a \emph{vertex subgraph of the transition graph}. We provide some basic definitions related to vertex subgraphs. These will allow us to explain the minor adjustments necessary to convert the proof of Theorem \ref{specrad} to a proof of Proposition \ref{eigoffuc}.  Further details can be found in \cite{eigoffuc}.

\subsection{The transition graph of $h$} As before, we can view $h$ as an element of $\tup{Aut}(\pi_1(\Sigma, *)) \cong \tup{Aut}(F_n)$. Let $X$ be a \emph{train track graph} for $h$, and $\varphi: X \to X$ be a train track representative (see \cite{BeH}), with base point $b \in X$. Pick an orientation on each edge of $X$.

Construct a directed graph $\CT$ called \emph{the transition graph of $h$} in the following way. The vertex set $V(\CT)$ is the set of edges of $X$. Let $e, e'$ be two vertices of $\CT$. Connect $e$ to $e'$ with $m$ edges emanating from $e$, where $m$ is the number of times the path $\varphi(e)$ passes through $e'$. 

Every edge $\eta$ of $\CT$ connecting $e$ to $e'$ corresponds to an $e'$-sub-edge of the path $\varphi(e)$. Let $\textbf{s}(\eta) = 1$ if this sub-edge is traversed in the positive direction and $\textbf{s}(\eta) = -1$ if it is traversed in the negative direction. Similarly, every path $\pi$ of length $k$ in $\CT$ connecting $e$ to $e'$ corresponds to an  $e'$-sub-edge of the path $\varphi^k(e)$. Denote $\textbf{s}(\pi) = 1$ if this sub-edge is traversed in the positive direction and $\textbf{s}(\pi) = -1$ if it is traversed in the negative direction.

If $\eta$ is an edge connecting $e$ to $e'$, we can write $\varphi(e)$ as a concatenation $\varphi(e) = a (e')^{\textbf{s}(\eta)} b$. Define the \emph{path corresponding to $\eta$} as $\textbf{p}(\eta) = a$ if $\textbf{s}(\eta) = 1$ and $\textbf{p}(\eta) = a (e')^{-1}$ otherwise. Similarly, we can define $\textbf{p}(\pi)$ where $\pi$ is a path in $\CT$ connecting $e$ to $e'$.

Let $G = \Gamma_0 \rtimes F_n$, and let $i: F_n \to G$ be the obvious map. Let $K_\Gamma$ be the kernel of the composition $F_n \to G \to H_1(G, \BQ)$. The cover of $X$ corresponding to $K_\Gamma$ is called the \emph{$\Gamma$-equivariant universal abelian cover of X}. We denote this cover $\wt{X}_\Gamma$.  Denote the deck group of this cover by $H$.  

Fix a preferred lift $\wt{b}$ of the base point $b$ to $\wt{X}_\Gamma$. For every vertex $c$ of $X$, pick a minimal length path connecting that vertex to $b$, and take a lift of that path to $\wt{X}_\Gamma$ originating at $\wt{b}$. Associate the terminal point of this path with the element $0 \in H$. Extend this $H$-equivariantly to an association of an element of $H$ to every vertex $v$ of $\wt{X}_\Gamma$. Call this association $\textbf{a}(v)$. Given a path $\pi$ in $\CT$, define the \emph{translation of $\pi$} to be $\textbf{t}(\pi) = \textbf{a}(\omega) - \textbf{a}(\alpha)$ where $\alpha$ and $\omega$ are the initial and terminal vertices respectively of a lift of  the path $\textbf{p}(\pi)$ to $\wt{X}_\Gamma$. Define the \emph{normalized translation of $\pi$} to be the element $\textbf{t}_n(\pi) = \frac{1}{\textup{length}(\pi)} \textbf{t}(\pi) \in H \otimes \BR$.

\subsection{Vertex subgraphs of the transition graph}

Construct the $m \times m$ matrix $A_\varphi \in M_m(\BC[H])$ by setting $$(A_\Gamma)_{i,j} = \sum_
\eta \textbf{s}(\eta) \textbf{t}(\eta) $$

\nid where the sum is taken over all edges $\eta$ connecting $e_i$ to $e_j$. For any $k$ we have that $$\trace A_\Gamma^k = \sum_\gamma \textbf{s}(\gamma) \textbf{t}(\gamma)$$

\nid where the sum is taken over all based cycles in $\CT$ of length k. Let $S_k \subset H \otimes \BR$ be $\frac{1}{k}$ times the support of $\trace{A}_\Gamma^k$. In \cite{eigoffuc}, we show that $S_k$ is contained in the convex hull of $\{\textbf{t}_n(\gamma)\}_\gamma$, where $\gamma$ ranges over all simple based cycles in $\CT$. The polygon given by the convex hull of these cycles is called the \emph{equivariant shadow of $\varphi$} and is denoted $\CS^e \varphi$. 

Given a vertex $v$ of the equivariant shadow, let $\CT_v$ be the union of all based cycles in $\CT$ whose normalized translation is $v$. In \cite{eigoffuc}, we show that $\textbf{t}_n(\gamma) = v$ if and only if $\gamma$ is a based cycle in $\CT_v$. The graph $\CT_v$ is called the \emph{vertex subgraph of $v$}. 

\subsection{Stabilizing vertex subgraphs}

A subgraph $\CT' \subset \CT$ is called \emph{stable} if $\sum_\gamma \textbf{s}(\gamma) \textbf{t}(\gamma) \neq 0$ for infinitely many values of $k$, where the sum is taken over all based cycles of length $k$ in $\CT'$. The vertex  $v$ of $\CS^e \varphi$ is called stable if $\CT_v$ is stable. Given a cover $p: X_0 \to X$ to which $\varphi$ lifts, the transition graph $\CT_0$ of a lift of $\varphi$ is a finite cover of $\CT$. The vertex $v$ is said to be \emph{stable in the cover $p$} if the lift of $\CT_v$ to $\CT_0$ is stable.

In Lemma 4.17 in \cite{eigoffuc}, we give similar definitions to the ones given above, except we take $G = \BZ \rtimes_h F_n$.  In this case, we prove that for every vertex $v$ of $\CS^e \varphi$, there is a finite cover where $v$ is stable. The construction provided in the proof of this lemma shows that this cover can be obtained by taking an iterated sequence of $\langle h^{m_i} \rangle$ admissible covers, where $m_i$ are sufficiently large integers. We wish to duplicate this lemma in our setting. 

\begin{lemma} \label{stabilization}
Let $v$ be a vertex of $\CS^e_\varphi$. Then there exists a sequence of finite covers $X_k \to \ldots \to X_1 \to X_0 = X$ such that the following hold: 
\begin{enumerate}
\item The group $\Gamma_0$ lifts to each cover. 
\item The cover $X_{i+1} \to X_i$ is $\Gamma_0$-admissible. 

\end{enumerate}

\end{lemma}

\begin{proof}

The only property of $G$ that is used in the proof of Lemma 4.17 in \cite{eigoffuc} is that $G$ is residually torsion free nilpotent. In \cite{resprop}, Koberda proves that for any $K < \tup{Aut}(F_n)$ such that $K_* \leq \tup{GL}(H_1(F_n, \BC))$ can be conjugated into $\CU$, the group $K \rtimes F_n$ is residually torsion free nilpotent. Thus, we can always find a cover $X' \to X$ in which $v$ is stable. In order to  complete the proof, we need to show that this cover can be taken to be of the form stated in the theorem. 

Denote $A = H_1(G, \BZ) / \tup{torsion}$. In our proof of $\BZ \rtimes F_n$ version of the lemma, we show that there is a number $k$ such that for almost any sequence of primes $p_1, \ldots, p_k$, the vertex $v$ is stable in the cover we get by taking iterated $p_i$ homology covers of $G$ (that is, passing to covers corresponding to the kernels of the maps $A/ pA$. 

Write $H_1(G, \BZ) = M \oplus N$ where $M$ is the abelianization of $\Gamma$, and $N$ is the image of $F_n$ in $H_1(G,\BZ)$. For any $p,q$ the kernel of the map $G \to M/qM \oplus N/pN$ is a semi direct product $\Gamma_q \rtimes (F_n)_p$ where $\Gamma_q$ is the kernel of the map $\Gamma \to H_1(\Gamma, \BZ/q\BZ)$ and $(F_n)_p$ is the kernel of the map $F_n \to H_1(F_n, \BZ/p\BZ)$.  We call this a $(p,q)$-cover. The same proof  as our proof in \cite{eigoffuc} (except for replacing $p$ covers with $(p,q)$ covers  above) shows that there is a number $k$ such that for almost any collection of primes $p_1, \ldots, p_k, q_1, \ldots, q_k$  the vertex $v$ is stable in the cover obtained by taking iterated $(p_i, q_i)$-covers. In this way, we get a sequence of covers $X_k \to \ldots X_1 \to X$ such that $v$ is stable in $X_k$. It remains to show that these covers are $\Gamma_0$-admissible. 

Since conjugation acts trivially on abelianization, $\Gamma_0$ acts trivialy on the deck group of $X_1 \to X$.  Let $F_{1,2} = ((F_n)_{p_1})_{p_2}$, and   $\Gamma_0 = \langle f_1, \ldots, f_s \rangle$. The elements $f_i$ act on $H_1(F_{1,2}, \BZ)$. For all sufficiently large primes $q$, the $f_i$-invariant vectors in $H_1(F_{1,2}, \BZ)$ are the same as the $f_i^q$ invariant vectors for every $i$. Thus, if we replace $q_2$ with almost any other $q$ (as we are allowed to do), we have that the $\Gamma_q$ invariant vectors in $H_1(F_{1,2}, \BZ)$ are the same as the $\Gamma$-invariant vectors. Thus, the group $\Gamma$ acts trivially on the deck group of $X_2 \to X_1$, and the cover is $\Gamma_0$-admissible. Proceeding inductively in this manner now gives the result.

\end{proof}

\subsection{Proof of Proposition \ref{eigoffuc}}

\begin{proof}
Our proof of Theorem \ref{specrad} uses only three properties of $\BZ \rtimes_h \pi_1(\Sigma, \BZ)$. 
\begin{enumerate}
\item The group $\BZ \rtimes_h \pi_1(\Sigma, \BZ)$ is residually torsion free nilpotent. This is used to prove the vertex stabilization lemma. 
\item The equivariant shadow $\CS^e h$ is a $m$ dimensional polygon, where $m = \dim H_1(\BZ \rtimes_f F_n/B, \BQ)$, and $B$ is the normal subgroup generated by boundary components. This is a theorem of Fried (\cite{Friedzeta}). 
\item If $h$ does not have eigenvalues off the unit circle, then for every $K > 0$ there exists an admissible cyclic cover $\Sigma' \to \Sigma$ where the dimension of the equivariant shadow is at least $K$. 
\end{enumerate}

\nid The first property holds in our case for every by Lemma \ref{stabilization}. To see the second property, note first that the map $F_n \to H_1(\Gamma_0 \rtimes F_n, \BZ)$ factors through the map $F_n \to H_1(\BZ\rtimes_h F_n, \BZ)$. Let $F' = F_n / B$. Given a subset $Y \subset H_1(\BZ_h \rtimes_h F', \BZ)$ that spans the image of $F_n$ in this group, the image of $Y$ in $H_1(\Gamma \rtimes F', \BZ)$ will also span the image of $F_n$.The second property now follows in our case from Fried's theorem. 

We now turn to the third property. Since $H_1(F_n, \BZ)$ has $\Gamma_0$-invariant vectors, we have that $H_1(G, \BC) \neq 0$. In particular, the image of $F_n$ in $H_1(G, \BZ) /\tup{torsion}$ has co-cyclic subgroups of arbitrarily large size. 

There is a splitting $H_1(F_n, \BZ) \cong M \oplus N$, where $M$ is the image of $B$ in this group. After passing to some $\Gamma_0$-admissible cover, we can assume that $\Sigma$ has at least $2$ boundary components, so $M \neq 0$. Pick a prime $p$ and a map  $H_1(G, \BZ) /\tup{torsion} \to \BZ / p\BZ$ that sends every element of $N$ to $0$ and each boundary component to $1$. This corresponds to a  cover  $ \Sigma' \to \Sigma$ with cyclic deck group $\BZ/ p\BZ$ such that $\Sigma'$ has the same number of boundary components as $\Sigma$.  

By Proposition \ref{prop1}, the group $(\Gamma_0)_* \in \tup{GL}(H_1(\Sigma', \BC))$ can be conjugated into $\CB$.Let $V$ be the image of the transfer map from $H_1(\Sigma, \BC) \to H_1(\Sigma', \BC)$. Let $W =  \ker H_1(\Sigma',\BC) \to H_1(\Sigma, \BC)$, where the map between homology groups is the map induced by the covering space map. There is a $\Gamma_0$-invariant splitting $H_1(\Sigma', \BC) \cong V \oplus W$. Picking bases for these spaces and conjugating into upper triangular matrices, we see that the group $(\Gamma_0)_*$ has common eigenvalues that are not $\BZ/p\BZ$-invariant. Let $U$ be the space of $[\Gamma_0, \Gamma_0]$-invariant vectors. This space is defined over $\BQ$, and is invariant under the action of the deck group. Decomposing $V$ as a sum of representations of $\BZ/p\BZ$ representations that are irreducible over $\BQ$, and noting that not all of these representations are $1$ dimensional, we get that $\dim V \geq p - 1$. Since $\Sigma'$ has the same number of boundary components as $\Sigma$, this gives the third property. 

\end{proof}

\section{Proof of Theorem \ref{theorem1}}

\nid We begin by proving a slightly more restrictive form of Theorem \ref{theorem1}. 

\begin{prop} \label{prop3}
	Suppose $\Gamma \leq \tup{Mod}(\Sigma)$ is a non-solvable subgroup such that $\Gamma$ contains two independent pseudo-Anosov elements, and such that $\Gamma_* \leq \tup{GL}(H_1(\Sigma, \BC))$ is virtually solvable. Then There exists a finite cover $p: \Sigma' \to \Sigma$ such that $\rho_p(\Gamma_0)$ is not virtually solvable. 
	
\end{prop}

\begin{proof}
	Since $\Gamma_*$ is virtually solvable, it contains a finite index subgroup that can be conjugated into $\CB$. Replace $\Gamma$ with this subgroup (note that the theorem still holds if we prove it for a non virtually solvable subgroup of $\Gamma$. 
	
	By a theorem of Fujiwara (\cite{Fujiwara}), if $f,g$ are two independent pseudo-Anosov mapping classes then $f^r g^s$ is pseudo-Anosov for all sufficiently large $r,s$. Let $\Gamma_0= [\Gamma,\Gamma]$. Fujiwara's theorem implies that $\Gamma_0$ contains two independent pseudo-Anosovs. Replace $\Gamma_0$ with some finitely generated subgroup generated by independent pseudo-Anosovs. Fujiwara's theorem implies that $[\Gamma_0, \Gamma_0]$ contains pseudo-Anosov mapping classes. Let $h$ be such an element. 
	
	The group $(\Gamma_0)_*$ can be conjugated into $\CU$. If Theorem \ref{theorem1} does not hold, then by Proposition \ref{eigoffuc}  there is a tower of $\Gamma_0$-admissible covers $\Sigma_k \to \ldots \to \Sigma_0 = \Sigma$ such that $\rho_{\Sigma_k \to \Sigma}(h)$ has eigenvalues off the unit circle. By repeated application of Proposition \ref{prop1}, the group $\rho_{\Sigma_k \to \Sigma}(\Gamma_0)$ can be conjugated into $\CB$. But $[\CB, \CB] = \CU$, which contradicts our assumption that $\rho_{\Sigma_k \to \Sigma}(h)$ has eigenvalues off the unit circle.

\end{proof}

\begin{proof}
We begin with several reduction steps that will allow us to apply Proposition \ref{prop3}. Every non virtually solvable subgroup of $\tup{Mod}(\Sigma)$ contains two mapping classes of positive entropy that generate a non virtually solvable group. Replace $\Gamma$ with such a subgroup. 

By passing to a finite index subgroup, we can assume that every element of $\Gamma$ is a pure mapping class - that is, for every $f \in \Gamma$  there is a multicurve  $\CC \subset \Sigma$ such that each connected component of $\Sigma \setminus \CC$ is $f$ invariant, and $f$ restricted to each such component is either trivial or a pseudo-Anosov mapping class. We call the collection of components where $f$ is pseudo-Anosov the \emph{pseudo-Anosov support of $f$}. 

Let $f,g$ be two pure mapping classes and suppose that $\Sigma_f \subset \Sigma$ is in the \emph{pseudo-Anosov} support of $f$, and $\Sigma_g \subset \Sigma$ is in the pseudo-Anosov support of $g$. Suppose further that $\Sigma_f \cap \Sigma_g$ is neither empty nor a union of annuli. The some element of $\langle f,g \rangle$ is a pure mapping class group whose pseudo-Anosov support contains a surface that contains $\Sigma_f \cup \Sigma_g$ (see Theorem 6.1 in Clay, Leininger, and Mangahas' paper \cite{clm}). This implies that there is an essential subsurface $S \subset \Sigma$ and a subgroup $\Gamma' < \Gamma$ such that $S$ is $\Gamma'$-invariant and $\Gamma'$ contains two elements whose restrictions to $S$ are two independent pseudo-Anosov mapping classes. Replace the group $\Gamma$ with $\Gamma'$. 
 
 By a theorem of Marshall Hall (\cite{Hall}), there is a finite cover $\wt{\Sigma} \to \Sigma$, and a lift $\wt{S}'$ of $S$ to $\wt{\Sigma}$ such that $\pi_1(\wt{S}, *)$ is a free factor of $\pi_1(\wt{\Sigma}, *)$, for some point $*$. A finite index subgroup of $\Gamma$ fixes this lift. Replace $\Gamma$ with this finite index subgroup. 
 
 Apply Proposition \ref{prop3} to the surface $\wt{S}$, together with the group $\Gamma|_{\wt{S}}$, to get a finite regular cover $p$ of $\wt{S}$ to which $\Gamma|_{\wt{S}}$ lifts, such that the image of $\Gamma|_{\wt{S}}$ under the corresponding homological representation is not virtually solvable. 
 
 Since $\pi_1(\wt{S}, *)$ is a free factor of $\pi_1(\wt{\Sigma})$, the cover $p$ can be extended to a cover of all of $\wt{\Sigma}$. Again, since $\pi_1(\wt{S}, *)$ is a free factor of $\pi_1(\wt{\Sigma})$, $H_1(p^{-1}(\wt{S}), \BC)$ injects into $H_1(p^{-1}(\wt{\Sigma}), \BC)$ as a $\rho_p(\Gamma)$-invariant subspace. Restriction to this subspace gives a non virtually solvable image of $\rho_p(\Gamma)$, and hence $\rho_p(\Gamma)$ itself is not virtually solvable.

\end{proof}

\vspace {10mm}


\end{document}